\documentclass [11pt]{amsart}
\usepackage{amsmath}
\usepackage{amssymb}
\usepackage{comment}
\usepackage{tikz-cd}
\usepackage{hyperref}
\usepackage{color}

\newtheorem{Thm}{Theorem}[section]
\newtheorem{rem}{Remark}[section]
\newtheorem{Cor}{Corollary}[section]
\newtheorem{Lem}{Lemma}[section]
\newcommand{\R}{\mathbb R}
\newcommand{\C}{\mathbb C}
\newcommand{\T}{\mathbb T}
\newcommand{\N}{\mathbb N}

\newcommand{\Z}{\mathbb Z}
\newcommand{\la}{\lambda}

\newcommand{\ra} {\rightarrow}
\newcommand{\txt}{\textmd}
\newcommand{\ds}{\displaystyle}
\newcommand{\be} {\begin{equation}}
	\newcommand{\ee} {\end{equation}}
\newcommand{\bes} {\begin{equation*}}
	\newcommand{\ees} {\end{equation*}}
\newcommand{\bea} {\begin{eqnarray}}
	\newcommand{\eea} {\end{eqnarray}}
\newcommand{\beas} {\begin{eqnarray*}}
	\newcommand{\eeas} {\end{eqnarray*}}
\usepackage{array}
\usepackage{tabularx,multicol}
\usepackage{enumerate}
\usepackage{mathrsfs}
\date{}
\usepackage[margin=1in]{geometry}
\numberwithin{equation}{section}

\begin{document}
	
	\title[\tiny{Completeness of Exponentials and Beurling's Theorem on $\R^n$ and $\T^n$}]{Completeness of Exponentials and Beurling's Theorem \\
	on $\R^n$ and $\T^n$}
	
	\author{\tiny{Santanu Debnath and Suparna Sen}}
	
	\address{Department of Pure Mathematics, University of Calcutta, India.}
	
	\email{santanudebnath1804@gmail.com, suparna29@gmail.com}
	
	\thanks{The first author is supported by CSIR Senior Research Fellowship (Enrollment Id: 09/028(1002)/2017-EMR-I)}

	
	\begin{abstract}
		A classical result of Arne Beurling states that the Fourier transform of a nonzero complex Borel measure $\mu$ on the real line cannot vanish on a set of positive Lebesgue measure if $\mu$ has certain decay. We prove a several variable analogue of Beurling's theorem by exploring its connection with the well-known problem concerning the density of linear span of exponentials in a certain weighted normed linear space of continuous functions. In the process, we also prove some new results of this genre and establish an equivalence between the above two problems. We also obtain a generalisation of Beurling's theorem and prove these results on the $n$-dimensional torus $\T^n.$
	\end{abstract}
	
	
	
	\maketitle
	
	\section{Introduction}
	
	Several classical results (referred as uncertainty principles in general) in harmonic analysis  deal with the phenomenon that a function on the real line and its Fourier transform cannot simultaneously be `small', for example if one decays rapidly at infinity then the other cannot vanish on a `large set' unless both vanish identically. As a manifestation of this fact, we note the following: if a function $f$ on $\R$ satisfies the estimate 
	$$ |f(x)| \leq e^{-|x|}, \quad \txt{ for all } x \in \R,$$
	and its Fourier transform $\widehat{f}$ vanishes on a set of positive Lebesgue measure, then $f$ is zero almost everywhere. In fact, here $\widehat{f}$ turns out to be holomorphic on a domain containing the real line because of the rapid decay on $f.$ The same conclusion holds if we interchange the conditions on $f \in L^1(\R)$ and $\widehat{f}$ due to the fact that $\R$ is its own unitary dual. So it is natural to study the interrelation of the optimal nature of the set on which one vanishes and the allowable decay of the other. 
	
	Plenty of classical literature are available on this, mainly due to Ingham \cite{I}, Paley-Wiener \cite{PW}, Levinson \cite{L1} and Beurling \cite{B} which studied results of this kind on the real line $\R.$ Among these results, the most general one (in terms of smallness of vanishing set) is due to Beurling \cite{B, K} \textcolor{black}{for complex (necessarily finite) Borel measures. We define the Fourier transform $\widehat{\mu}$ of a complex measure $\mu$ on $\R$ (in the classical sense) as a continuous function by
	$$\widehat{\mu}(\la)=\int_{\R} e^{-i\la t}d\mu(t), \text{ for } \la \in \R.$$}
	\begin{Thm}[Beurling]\label{ber-th}
		Let $\mu$ be a complex \textcolor{black}{(finite)} Borel measure on $\R$ such that 
		\begin{equation}\label{eq;1}
			\int_0^{\infty}\frac{\textcolor{black}{-\log|\mu|([x,\infty))}}{1+x^2}dx=\textcolor{black}{+}\infty,
		\end{equation}
 where $|\mu|$ is the total variation of the measure $\mu$. If the Fourier transform $\widehat{\mu}$ of $\mu$ vanishes on a set $\Lambda$ of positive Lebesgue measure in $\R,$ then $\mu $ is identically zero.
	\end{Thm}
\textcolor{black}{ \begin{rem}
Since $\widehat{\mu}$ is continuous, $\widehat{\mu}$ vanishes on $\Lambda$ if and only if it vanishes on $\overline{\Lambda}$ as well. So, we will henceforth assume that the vanishing set $\Lambda$ is closed.  
	\end{rem} }
 
	This was proved in \cite{B} as a consequence of a characterization of the Beurling quasianalytic class (a generalisation of the famous Denjoy-Carleman quasianalytic class and Bernstein quasianalytic class) using the concept of harmonic measure. A more direct proof using similar ideas is available in \cite{K}. Moreover, there is another quasianalyticity result by Beurling where the proof uses similar method but involves more technical dificulties (see \cite{B, K}), from which the following result is obtained:
 	\begin{Thm}[Beurling]\label{ber-th-c}
		Let $f\in L^2(\T)$ satisfy
		\begin{equation}\label{eq;7c}
			\sum_{k=0}^\infty\dfrac{1}{1+k^2} \log\left(\dfrac{1}{\sum_{m=k}^\infty |\widehat{f}(m)|^2}\right)=\infty,
		\end{equation}
		where $\widehat{f}(k)$ denote the Fourier coefficients of $f.$ If $f$ vanishes on a set of positive Lebesgue measure in $\T,$ then $f$ is zero almost everywhere.
	\end{Thm}
	
	Recently there has been a plethora of results proving analogues of the ones due to Ingham, Levinson and Paley-Wiener for several Lie groups and homogeneous spaces including the $n$-dimensional Euclidean space, $n$-dimensional torus, two step nilpotent Lie groups, Euclidean motion group, Riemannian symmetric spaces etc (see \cite{Bh2, Bh, BPP, BP, BR, BRS, BS1, BS2, GT, GT3}). However, analogues of the result due to Beurling (Theorem \ref{ber-th}) remained largely unexplored except few results regarding some integral transforms (proved using Beuring's original idea) in a very recent article \cite{DS1}. In fact, Beurling's result is a straightforward generalisation of Levinson's result, which considers functions vanishing on open sets instead of sets of positive Lebesgue measure. 
	
	We recall that the approach to the extension of Levinson's theorem on $\R^n$ (see \cite{BRS}) and some other noncommutative setting (see \cite{BR}) relies heavily on an alternative proof suggested in \cite{K} where it was shown that the theorem of Levinson can also be obtained as a consequence of completeness of linear span of exponentials in certain normed linear \textcolor{black}{spaces} of continuous functions. Several versions of the problem regarding completeness of exponentials and its multiple reformulations have been studied throughout the last century by many prominent mathematicians including Akhiezer, Bernstein, de Branges, Krein, Koosis, Levinson, Mergelyan and many others. This is still a very active area of research with more recent significant contributions by Bakan \cite{Ba}, de Jeu \cite{DJ}, Poltoratski \cite{P1,P} among others. Besides the inherent beauty of the original problem, such an extensive interest is largely because of its numerous links with other areas of classical analysis such as spectral problems for differential operators, gap and density problems, type problem etc. We note that Beurling's result (as restated in \textcolor{black}{Lemma} \ref{psi-ber}) can be considered as a direct generalisation of de Branges' Gap Theorem (see \textcolor{black}{\cite[Theorem 63]{dB}, \cite[Theorem 9]{P1}}). 
	
	In the study of exponential density problem on $\R^n,$ we start with a weight function $\psi:[0,\infty) \rightarrow [0,\infty)$ such that $\psi(x)\to \infty$ as $x\to \infty.$ We then define a weighted space of continuous functions $$C_{\psi}(\R^n)=\left\lbrace f : \R^n \to \C ~ : ~ f \txt{ is continuous and } \lim_{|x|\to \infty}\dfrac{f(x)}{e^{\psi(|x|)}}=0\right\rbrace,$$
	equipped with the weighted uniform norm 
	$$\|f\|_{\psi}=\sup_{x \in \R^n}\dfrac{|f(x)|}{e^{\psi(|x|)}},\quad \txt{ for } f\in C_{\psi}(\R^n).$$ For $\Lambda \subset \R^n,$ we consider the linear span of the exponential functions given by 
	$$\Phi_{\Lambda}(\R^n)= span \lbrace e_{\la}:\la \in \Lambda \rbrace,$$ where $e_{\la}(x)=e^{-i\la \cdot x}$ for $x \in \R^n.$ A version of the exponential density problem deals with the conditions on the set $\Lambda\subset \R^n$ and \textcolor{black}{the} weight $\psi,$ for which the space of exponentials $\Phi_{\Lambda}(\R^n)$ is dense in $C_{\psi}(\R^n).$ Another version of this problem is to study density using $L^p$ norms with respect to \textcolor{black}{$|\mu|,$ for a complex (finite) measure $\mu.$} In this version, one studies conditions on $\Lambda\subset \R^n$ and $\mu$ that ensure completeness, that is, density of $\Phi_{\Lambda}(\R^n)$ in $L^p(\R^n,\textcolor{black}{|}\mu\textcolor{black}{|}).$ For some results of this type, see  \cite{Ba, DJ, K, P}. In most of these results, $\Lambda$ is considered to be an open set. 
	
	In this paper, we have been able to prove a stronger (than those available in literature) exponential density result (Theorem \ref{density}) for a \textcolor{black}{closed set $\Lambda$} of positive Lebesgue measure in $\R,$ as a consequence of Beurling's result on $\R$ (Theorem \ref{ber-th}). This also enables us to prove analogous results for completeness of exponentials on $\R^n$ (Theorem \ref{several density}, Theorem \ref{Lp dense}) using standard techniques. Thus we are able to prove a several variable analogue of Beurling's Theorem on $\R^n$ (Theorem \ref{several beurling}). Hence we can establish the following equivalence between the above two problems of different genre, namely the density of exponentials and Beurling's theorem, via the integrability condition on $\psi.$ 

	\begin{Thm}\label{equiv}
		Let $\psi:[0,\infty)\rightarrow [0,\infty)$ be a continuous increasing function such that $\psi(x)\to \infty $ as $x\to \infty.$ Then the following are equivalent:
		\begin{itemize}
			\item[(1)] $\displaystyle \int_0^\infty \dfrac{\psi(x)}{1+x^2}dx =\infty.$
			\item[(2)] $\Phi_\Lambda(\R^n)$ is dense in $(C_\psi(\R^n),|\cdot\|_\psi),$ \textcolor{black}{if $\Lambda \subset \R^n$ is a product of one dimensional closed sets of positive Lebesgue measure.}
           \textcolor{black}{\item[(3)] If $\mu $ is a complex Borel measure on $\R^n$ such that 
			\begin{equation*}
			\int_{\R^n}e^{\psi(|x|)}d|\mu|(x)<\infty,
			\end{equation*}
			and $\widehat{\mu}$ vanishes on a product of one dimensional
closed sets of positive Lebesgue measure, then $\mu$ is identically zero.}
		\end{itemize}
	\end{Thm}
	
	Furthermore, as a consequence of \textcolor{black}{Theorem \ref{several beurling}}, we derive an unusual uncertainty principle (\textcolor{black}{Corollary} \ref{Beurling lp}), which turns out to be a generalisation of Beurling's theorem. In this context, we note that, Theorem \ref{several density} may be considered as a generalisation of both Theorem 4.13 (for the admissible space in Definition 4.11, 3(b)) of \cite{DJ} and Lemma 2.4 of \cite{BRS}. 
	
	Similar questions may arise in the context of the torus $\T^n.$ The standard technique (for example see \cite{BRS}) here is to reduce matters to the corresponding results on $\R^n.$ This reduction involves a convolution of the given function with a smooth function supported on a small compact set, which can be adjusted so that the vanishing set of the resulting convolution is also an open set. However, this convolution technique doesn't readily work in the case of functions whose vanishing sets are of positive Lebesgue measure. So we require to follow the technique of proof used in the case of $\R^n$ for $\T^n$ with some necessary modifications to obtain Theorem \ref{sev-den-c}. In fact, this also enables us to extend the result to the space of integrable functions (Theorem \ref{sev-fn-c}) from the space of square integrable functions in Theorem \ref{ber-th-c}.
	
	We will adopt the following notation and convention throughout this paper: 
	The Lebesgue measure of a set $E\subset \R$ is denoted by $m(E)$ but for convenience $dx$ will be used instead of $dm(x).$ For a Banach space $X$ endowed with a norm $\|\cdot\|,$ the dual of $X$ is defined to be the set of all bounded linear functionals on $X$ denoted by $(X,\|\cdot\|)^*.$ The space of all smooth functions on $X$ is denoted by $C^\infty(X)$ and $C_c^\infty(X)$ denotes the space of all smooth, compactly supported functions on $X.$ The open ball of radius $l>0$ centered at $a$ in $\R^n$ is denoted by $B(a,l).$ For $x, y \in \R^n,$ $|x|$ is used to denote the Euclidean norm of the vector $x$ and $x \cdot y$ to denote the Euclidean inner product of the vectors $x$ and $y.$
	
	\section{On the Euclidean space $\R^n$}
	This section is broadly divided into two subsections. In the first subsection our aim is to prove some results on density of exponentials in certain Banach spaces of functions on $\R^n$ (Theorem \ref{several density} and Theorem \ref{Lp dense}). We will use Beurling's result on $\R$ (Theorem \ref{ber-th}) for proving these results. In the second subsection, we will use the exponential density result from the first subsection to prove a several variable analogue (Theorem \ref{several beurling}). \textcolor{black}{Finally, we will apply Theorem \ref{several beurling} to prove a generalisation (\textcolor{black}{Corollary}  \ref{Beurling lp}) of Beurling's result.} 
	
	First we will restate Beurling's result in a way similar to Theorem 1.1 of \cite{BRS}. 
	\textcolor{black}{\begin{Lem} \label{psi-ber}
		Let $\psi:[0,\infty)\rightarrow [0,\infty)$ be an increasing function such that
  \begin{equation}\label{psi cond}
			\int_0^{\infty}\dfrac{\psi(x)}{1+x^2}dx=\infty.
		\end{equation}  
  If $\mu$ is a complex (finite) Borel measure on $\R$ such that
		\begin{equation}\label{mu decay}
			\int_{0}^\infty e^{\psi(x)}d|\mu| (x)<\infty,
		\end{equation}
		then $\mu$ satisfies (\ref{eq;1}), in particular Theorem \ref{ber-th} holds for $\mu.$ 
	\end{Lem}}
    \begin{proof} Since  $\psi$ is increasing on $[0,\infty)$ we have for all $x\geq 0,$
		$$e^{\psi(x)}\int_x^\infty d\textcolor{black}{|}\mu\textcolor{black}{|}(t) \leq \int_{0}^\infty e^{\psi(t)}d\textcolor{black}{|}\mu\textcolor{black}{|}(t).$$
		From (\ref{mu decay}), it follows that for all $x \geq 0,$ there exists a constant $C>0$ such that 
		$$\int_x^{\infty} d\textcolor{black}{|}\mu\textcolor{black}{|}(t) \leq Ce^{-\psi (x)},$$
		which implies that
	$$\int_0^{\infty}\dfrac{1}{1+x^2}\log\left(\dfrac{1}{\int_x^{\infty}d\textcolor{black}{|}\mu\textcolor{black}{|}(t)}\right)dx \geq C\int_0^{\infty}\dfrac{\psi(x)}{1+x^2}dx.$$
		So by equation (\ref{psi cond}), we conclude that \bes\int_0^{\infty}\dfrac{1}{1+x^2}\log\left(\dfrac{1}{\int_x^{\infty}d\textcolor{black}{|}\mu\textcolor{black}{|}(t)}\right)dx=\infty.\ees
  \textcolor{black}{So $\mu$ satisfies  (\ref{eq;1}). Hence Theorem \ref{ber-th} holds for $\mu.$ }
  \end{proof}

\subsection{Density of Exponentials}
	
	We consider the space of exponentials $\Phi_{\Lambda}(\R^n)$ and the weighted space of continuous functions $C_\psi(\R^n)$ defined in the introduction. Since we assume that $\psi(x)\to \infty $ as $x\to \infty,$ it is easy to see that for any $\Lambda \subset \R^n,$ $\Phi_{\Lambda}(\R^n)$ is a subspace of $C_{\psi}(\R^n).$ For a \textcolor{black}{complex (finite)} measure $\mu$ on $\R^n,$ $\Phi_{\Lambda}(\R^n)$ is also a subspace of $L^p(\R^n, \textcolor{black}{|}\mu\textcolor{black}{|}),$ the space of all $L^p$ functions on $\R^n$ with respect to the measure $\textcolor{black}{|}\mu\textcolor{black}{|}.$ We will prove two results regarding the density of $\Phi_{\Lambda}(\R^n)$ in $(C_{\psi}(\R^n),\|\cdot\|_\psi)$ (Theorem \ref{several density}) and in $L^p(\R^n,\textcolor{black}{|} \mu\textcolor{black}{|})$ (Theorem \ref{Lp dense}) in this subsection.
	
	We will calculate $(C_{\psi}(\R^n),\|\cdot \|_{\psi})^*$ explicitly for continuous $\psi$ in order to prove our result. Let $C_0(\R^n)$ denote the Banach space of all continuous functions on $\R^n$ vanishing at infinity with respect to the supremum norm $\|\cdot\|_{\infty}$ and $\mathcal{M}(\R^n)$ denote the collection of all complex Borel measures on $ \R^n,$ which is the dual of $(C_0( \R^n),\|\cdot\|_{\infty}) .$ \textcolor{black}{For $f \in C_0( \R^n),$ using the continuity of $\psi,$ we define a bijective linear isometry $A_\psi : (C_0( \R^n),\|\cdot\|_{\infty}) \ra (C_{\psi}( \R^n),\|\cdot\|_{\psi})$ by $A_\psi(f)(x) = f(x)e^{\psi(|x|)},$ for $x \in  \R^n.$ From this, we get a natural bijection between the respective dual spaces, which enables us to get the following lemma:}
	\begin{Lem}\label{dual lem}
		If $\psi$ is a non-negative continuous function on $[0,\infty),$ then $(C_{\psi}( \R^n),\|\cdot\|_{\psi})$ is a Banach space isometrically isomorphic to $(C_0( \R^n),\|\cdot\|_{\infty})$ and its dual is given by
		$$(C_{\psi}( \R^n),\|\cdot\|_{\psi})^*=\left\lbrace\beta \in \mathcal{M}( \R^n) : 
		\int_{ \R^n}e^{\psi(|x|)}d\textcolor{black}{|}\beta\textcolor{black}{|}(x)<\infty \right\rbrace .$$
	\end{Lem}
	\textcolor{black}{Now,} we will use \textcolor{black}{above lemma and Lemma} \ref{psi-ber} to prove the following weighted approximation result on $\R:$
	\begin{Thm}\label{density}
		Let $\psi:[0,\infty) \rightarrow [0,\infty)$ be a continuous, increasing function satisfying 
		\bes\int_0^{\infty}\dfrac{\psi(x)}{1+x^2}dx=\infty. \ees
		For any \textcolor{black}{closed set $\Lambda$ of positive Lebesgue measure in $\R,$ } $\Phi_{\Lambda}(\R)$ is dense in $(C_{\psi}(\R),\|\cdot \|_{\psi}).$
	\end{Thm}
	\begin{proof}
		Since $\Phi_{\Lambda}(\R)$ is a subspace of $C_{\psi}(\R),$ in order to prove the result, it is enough to show that for any bounded linear functional $T$ on $(C_{\psi}(\R),\|\cdot \|_{\psi}),$ if $T$ vanishes on the space $\Phi_{\Lambda}(\R),$ then $T$ is identically zero. Let us consider such a $T \in (C_{\psi}(\R),\|\cdot \|_{\psi})^*.$ From Lemma \ref{dual lem}, we get a complex  Borel measure $\beta$ on $\R$ such that
		$$T(f)=\int_\R f(t) ~ d\beta(t), \quad \txt{ for } f \in C_{\psi}(\R),$$
		where $\beta$ satisfies $$\int_{ \R}e^{\psi(|x|)}d\textcolor{black}{|}\beta\textcolor{black}{|}(x)<\infty.$$
		Since $T$ vanishes on the space $\Phi_{\Lambda}(\R),$ we get that  
		$$ \int_\R e^{-i\lambda t}d\beta(t)=0, \quad \quad ~\forall ~ \lambda \in \Lambda.$$
		This implies that the function $\hat{\beta}$ vanishes on the set $\Lambda \subset \R.$ So $\beta$ satisfies all the conditions of \textcolor{black}{Lemma} \ref{psi-ber}, from which it follows that $\beta$ is identically zero. Hence we conclude that $ T \equiv 0.$
	\end{proof}
	
	We wish to prove an analogue of the above result on $\R^n,$ in a way similar to Lemma $2.4$ in \cite{BRS} which involves open sets in $\R^n.$ We note that an open set $U\subseteq \R^n$ always contains a set of the form $U_1\times U_2 \times \cdots \times U_n$ where each $U_j \subseteq \R$ is open in $\R$ for $1\leq j \leq n.$ This is an important tool in the technique of the proof of Lemma $2.4$ in \cite{BRS}. However, this property does not hold for \textcolor{black}{closed} sets of positive Lebesgue measure in $\R^n$. For example in $\mathbb{R}^2$ consider the set 
	$$E=\lbrace (x,y)\in [0,1]\times [0,1] : x-y \in F \rbrace ,$$ where $F\subset [0,1]$ is a fat Cantor set, which is a Cantor-like set obtained by removing middle intervals of length $(\frac{1}{4})^n$ from $[0,1]$ for each $n$-th iteration. Since $F$ is a closed set of positive Lebesgue measure in $\R,$ it follows that $E$ is a closed set of positive Lebesgue measure in $\R^2.$ It is easy to see that $E$ cannot contain $\Lambda_1 \times \Lambda_2$ such that $\Lambda_1, \Lambda_2$  \textcolor{black}{are closed subsets of $\mathbb{R}$ with $ m(\Lambda_1), ~ m(\Lambda_2)>0.$ Otherwise, $\Lambda_1 - \Lambda_2$ must contain an interval,} which contradicts the definition of $E$ as $F$ contains no interval. It is due to this reason that we are led to assume the vanishing set of the function to be \textcolor{black}{a product of one dimensional closed sets of positive Lebesgue measure} in order to incorporate the above property. In the following result, we will prove that if $\Lambda \subset \R^n$ is of \textcolor{black}{the form $\Lambda = \Lambda_1 \times \cdots \times \Lambda_n,$ where $\Lambda_j$ is a closed subset of $\R$ for each $1 \leq j \leq n$ such that $m(\Lambda_j)>0$}, then the density of $\Phi_{\Lambda}(\R^n)$ in $(C_{\psi}(\R^n),\|\cdot \|_{\psi})$ is characterized by the given integral condition on $\psi:$  
	
	\begin{Thm}\label{several density}
		Let $\psi$ be a non-negative, continuous, increasing function on $[0,\infty)$ such that $\psi(x)\to \infty $ as $x\to \infty.$ The space $\Phi_{\Lambda}(\R^n)$ is dense in $(C_{\psi}(\R^n),\|\cdot \|_{\psi})$ for \textcolor{black}{$\Lambda = \Lambda_1 \times \cdots \times \Lambda_n,$ where $\Lambda_j$ is a closed subset of $\R$ for each $1 \leq j \leq n$ such that $m(\Lambda_j)>0$} if and only if 
		\be \label{psiint} \int_0^{\infty}\dfrac{\psi(x)}{1+x^2}dx=\infty. \ee
	\end{Thm}
	 
\begin{proof} 
		First, we will assume (\ref{psiint}) and prove the density of $\Phi_{\Lambda}(\R^n)$ in $(C_{\psi}(\R^n),\|\cdot \|_{\psi}).$
		Let us define, 
		$$\psi_0 (x)= \dfrac{\psi (x)}{n}, \quad \quad \txt{for } x\in [0,\infty).$$ 
		\textcolor{black}{We now consider the following space of functions
		\bes 
		\mathcal{P}C_{\psi_0}(\R^n) = \text{span} \left\{f:\R^n \to \C :  f(x_1, \cdots, x_n)=f_1(x_1)\cdots f_n(x_n), f_j\in C_{\psi_0}(\R), 1\leq j \leq n  \right\}. 
		\ees
		It was shown in \cite{BRS} that $ \mathcal{P}C_{\psi_0}(\R^n)$ is dense in $(C_\psi(\R^n),\|\cdot \|_\psi).$ 
		Since $\Phi_{\Lambda}(\R^n) \subseteq (C_{\psi}(\R^n),\|\cdot \|_{\psi}),$ in order to prove the result it is enough to prove that $\Phi_{\Lambda}(\R^n)$ is dense in $ (\mathcal{P}C_{\psi_0}(\R^n),\|\cdot \|_\psi).$ We consider functions in $\mathcal{P}C_{\psi_0}(\R^n)$ of the form $$f(x)= f_1(x_1)\cdots f_n(x_n),$$ for $f_j\in C_{\psi_0}(\R),$ for all $1\leq j \leq n.$ Since $m(\Lambda_j)>0$ for each $j,$ for any $0<\epsilon<1$ applying Theorem \ref{density} we get $g_j\in \Phi_{\Lambda_j}(\R)$ such that for all $1\leq j\leq n$
		\begin{equation*}
			\sup_{s\in \R}\frac{|f_j(s)- g_j(s)|}{e^{\psi_0(|s|)}}<\epsilon .
		\end{equation*}
		By triangle inequality, we have
		$$ \sup_{s\in \R} \dfrac{|g_j(s)|}{e^{\psi_0(|s|)}}\leq 1+\|f_j\|_{\psi_0}, \quad \txt{ for } 1\leq j \leq n. $$ 
		Since $e_{\lambda_1}(x_1)\cdots e_{\lambda_n}(x_n)=e_{\lambda}(x),$ for $\la = (\la_1, \cdots, \la_n) \in \Lambda_1 \times \cdots \times \Lambda_n$  and $x = (x_1, \cdots, x_n) \in \R^n,$ it easily follows that if $$g(x)= g_1(x_1)\cdots g_n(x_n),\quad \txt{ for }  x=(x_1, \cdots , x_n)\in \R^n,$$ then $g \in \Phi_\Lambda(\R^n).$} By defining $$g_0(y)=e^{\psi_0(|y|)}=f_{n+1}(y),~~ y\in \R,$$ we have for all $x=(x_1,\cdots , x_n)\in \R^n,$
		\begin{eqnarray*}
			\frac{|f(x)- g(x)|}{e^{\psi(|x|)}} &\leq& \frac{|f_1(x_1)\cdots f_n(x_n)- g_1(x_1)\cdots g_n(x_n)|}{e^{\psi_0(|x_1|)}\cdots e^{\psi_0(|x_n|)}}\\
			&\leq & \sum_{k=1}^n\dfrac{|f_k(x_k)-g_k(x_k)|}{e^{\psi_0(|x_k|)}}\left(   \prod_{j=k+1}^{n+1}\dfrac{|f_j(x_j)|}{e^{\psi_0(|x_j|)}}\prod_{j=0}^{k-1}\dfrac{|g_j(x_j)|}{e^{\psi_0(|x_j|)}} \right) \\ 
			&\leq & \epsilon ~ n \prod_{j=1}^n (1+\|f_j\|_{\psi_0}) \\ 
			&\leq & C\epsilon.
		\end{eqnarray*}
		This proves the first part of the result.

  For the converse part, let us assume that $\Phi_{\Lambda}(\R^n)$ is dense in $(C_{\psi}(\R^n),\|\cdot \|_{\psi})$ for any set $\Lambda \subset \R^n$ \textcolor{black}{which is a product of one dimensional closed sets of positive Lebesgue measure}. If possible, let $$\int_0^{\infty}\dfrac{\psi(x)}{1+x^2}dx<\infty.$$ By Theorem 2.6(b) of \cite{BRS}, we get a non-zero $f\in C_c^\infty(\R^n)$ satisfying $$|\widehat{f}(\xi)|\leq Ce^{-\psi(|\xi|)}\quad  \text{ for all } \xi\in \R^n.$$
		Again, applying Lemma 2.5 of \cite{BRS}, there exists a non-zero continuous function $F$ on $\R^n$ which vanishes on a non-empty open subset of $\R^n$ and satisfies
		\be\label{F cond}
		\int_{\R^n}|\widehat{F}(\xi)|e^{\psi(|\xi|)}d\xi<\infty.\ee
		We define a non-zero complex Borel measure $\mu$ on $\R^n$ by 
		\bes \label{mu} d\mu(x)=\widehat{F}(x)dx. \ees 
		By Lemma \ref{dual lem}, it follows from (\ref{F cond}) that $\mu\in (C_\psi(\R^n),\|\cdot\|_\psi)^*.$ Moreover, since both $\widehat{\mu}$ and $F$ are continuous functions on $\R^n,$ $\widehat{\mu}$ also vanishes on an open set, say $U \subset \R^n.$
		This implies that as a bounded linear functional on $(C_\psi(\R^n),\|\cdot\|_\psi),$ $\mu$ vanishes on $\Phi_U(\R^n).$ Since $U$ \textcolor{black}{contains a product of one dimensional closed sets of positive Lebesgue measure,} using the density of  $\Phi_U(\R^n)$ in $(C_\psi(\R^n),\|\cdot\|_\psi),$ we can conclude that $\mu$ is identically zero, which gives a contradiction. Hence the result is proved.
	\end{proof}

	Now we will prove the density of $\Phi_\Lambda(\R^n)$ in another class of Banach space, $L^p(\R^n,\textcolor{black}{|}\mu\textcolor{black}{|})$ for $1\leq p < \infty,$ the space of $L^p$ functions with respect to $\textcolor{black}{|}\mu\textcolor{black}{|},$  where $\mu$ is a \textcolor{black}{complex (finite)} measure satisfying certain conditions. This is actually a consequence of the above result (Theorem \ref{several density}) along with a well-known result by A. Bakan (see \cite{Ba}, \textcolor{black}{\cite[Theorem 14]{P1}}). For the sake of completeness we will provide a direct proof. 

	\begin{Thm}\label{Lp dense}
Let $\psi:[0,\infty)\rightarrow [0,\infty)$ be an increasing (not necessarily continuous) function  such that 
		\be\label{a} \int_0^\infty\dfrac{\psi(x)}{1+x^2}dx=\infty \ee and $\mu$ be a \textcolor{black}{complex (finite)} measure satisfying 
		\be\label{b} \int_{\R^n} e^{\psi(|x|)}d\textcolor{black}{|}\mu\textcolor{black}{|}(x)<\infty. \ee
		For any set  $\Lambda \subset \R^n,$ \textcolor{black}{which is a product of one dimensional closed sets of positive Lebesgue measure}, the space $\Phi_\Lambda(\R^n)$ is dense in $L^p(\R^n,\textcolor{black}{|}\mu\textcolor{black}{|}),~ 1\leq p < \infty.$ 
\end{Thm}
\begin{proof}
		Since $C_c(\R^n)$ is dense in $L^p(\R^n,\textcolor{black}{|}\mu\textcolor{black}{|}),$ for $1\leq p < \infty$ it is enough to prove that $\Phi_\Lambda(\R^n)$ is dense in $C_c(\R^n)$ with respect to the corresponding $\|\cdot\|_p$ norm of $L^p(\R^n,\textcolor{black}{|}\mu\textcolor{black}{|}),$ for $1\leq p < \infty.$ We consider the continuous, increasing function $\Psi:[0,\infty)\rightarrow [0,\infty)$ defined by
	\bes\Psi(x) = \begin{cases} \frac{1}{p}\int_{x-1}^x \psi(t)dt, & x \geq 1 \\ 0, & x<1 \end{cases}.\ees
	 Now, as $p \, \Psi(x)\leq \psi(x),$ by (\ref{b}) we have
	\be\label{c} \int_{\R^n} e^{p\Psi(|x|)}d\textcolor{black}{|}\mu\textcolor{black}{|}(x)<\infty. \ee
	Again, since  $\psi(x-1)\leq p\,\Psi(x),$ it easily follows from (\ref{a}) that 
	\bes \int_0^\infty\dfrac{\Psi(x)}{1+x^2}dx=\infty. \ees
Let $f\in C_c(\R^n) \subseteq C_{\Psi}(\R^n).$ Given any $\epsilon >0,$ from Theorem \ref{several density} we get $g\in \Phi_\Lambda(\R^n)$ such that $\|f-g\|_{\Psi}<\epsilon.$ From (\ref{c}) we get $C>0$ such that 
	\bes\|f-g\|_p^p ~=~ \int_{\R^n}\dfrac{|f(x)-g(x)|^p}{e^{p\Psi(|x|)}}e^{p\Psi(|x|)}d\textcolor{black}{|}\mu\textcolor{black}{|}(x) ~\leq~ \|f-g\|_{\Psi}^p \int_{\R^n} e^{p\Psi(|x|)}d\textcolor{black}{|}\mu\textcolor{black}{|}(x) ~<~ C \epsilon^p.\ees	
	Hence the theorem follows.
\end{proof}

	\begin{rem}
		It follows from the proof of Theorem 14 of \cite{P1} that if $\Phi_{\Lambda}(\R^n)$ is dense in $L^p(\R^n,\textcolor{black}{|}\mu\textcolor{black}{|}),$ for a \textcolor{black}{complex (finite)} measure $\mu$ on $\R^n,$ then there exists $\psi$ satisfying (\ref{b}) such that $\Phi_{\Lambda}(\R^n)$ is dense in $(C_\psi(\R^n),\|\cdot\|_\psi).$
	\end{rem}
	
	\subsection{Beurling's Theorem}
	We will now prove some several variable analogues of Beurling's result (\textcolor{black}{Lemma} \ref{psi-ber}) including a generalisation of the same. \textcolor{black}{The proof of Theorem \ref{several beurling} is} based on the ideas used in \cite{BRS} to prove the analogue of Levinson's theorem. Accordingly, a main ingredient of the \textcolor{black}{proof is} the density of a linear span of exponentials $\Phi_{\Lambda}(\R^n)$ in the Banach \textcolor{black}{space} $(C_{\psi}(\R^n), \|\cdot\|_{\psi}),$ which we have already proved in Theorem \ref{several density}. \textcolor{black}{Thereafter, we prove Corollary \ref{Beurling lp}, which turns out to be a generalisation of Beurling's Theorem.} We have the following analogue of \textcolor{black}{Lemma} \ref{psi-ber} on $\R^n$:
	
	\begin{Thm}\label{several beurling}
		Let $\psi:[0,\infty) \rightarrow [0,\infty)$ be an increasing (not necessarily continuous) function such that $\psi(x)\to \infty $ as $x \to \infty$ and
		\begin{equation*}
			I = \int_0^\infty \dfrac{\psi(x)}{1+x^2}dx.
		\end{equation*}
		\begin{itemize}
			\item[(a)] \textcolor{black}{If $I= +\infty$ and $\mu$ is a complex (finite) Borel measure on $\R^n$ satisfying 
			\begin{equation}\label{eq;4}
				\int_{\R^n}e^{\psi(|x|)}d|\mu|(x)<\infty
			\end{equation}
			such that } $\widehat{\mu}$ vanishes on a set $\Lambda \subset \R^n,$ \textcolor{black}{which is a product of one dimensional closed sets of positive Lebesgue measure}, then $\mu $ is identically zero.
			\item[(b)] \textcolor{black}{If $I < +\infty,$} then there exists a non-trivial complex Borel measure $\mu$ on $\R^n$ satisfying (\ref{eq;4}) such that $\widehat{\mu}$ vanishes on \textcolor{black}{ a product of one dimensional closed sets of positive Lebesgue measure}.	
		\end{itemize}
	\end{Thm}
	
	\begin{proof}
		First we shall prove (a) by further assuming $\psi$ is continuous. Otherwise, we can work with the non-negative, continuous, increasing function $\Psi$ (for $p=1$), as in the proof of Theorem \ref{Lp dense}. Since the complex Borel measure $\mu$ satisfies (\ref{eq;4}), from Lemma \ref{dual lem} we have that $\mu \in (C_\psi(\R^n),\|\cdot\|_\psi)^*,$ that is,  there is a bounded linear functional $T_{\mu}$ on $C_{\psi}(\R^n)$ defined by
		$$T_\mu(g)=\int_{\R^n}g(x)d\mu(x), \quad \txt{for } g\in C_{\psi}(\R^n).$$  
		Since $\widehat{\mu}$ vanishes on $\Lambda \subset \R^n,$ we have 
		$$T_\mu(e_\la) = \widehat{\mu}(\la) = 0, \quad \txt{ for } \la \in \Lambda.$$ Moreover, as $\Phi_{\Lambda}(\R^n)$ is spanned by such $e_\la$'s, it follows that $$T_\mu(\phi) = 0, \quad \txt{ for } \phi \in \Phi_{\Lambda}(\R^n).$$
		Since $\Lambda \subset \R^n$ is \textcolor{black}{ a product of one dimensional closed sets of positive Lebesgue measure}, applying Theorem \ref{several density} we get that $T_\mu$ is identically zero on $C_{\psi}(\R^n).$ Hence we can conclude that $\mu$ is identically zero. This proves part (a).
		
		Now we shall prove (b). If $I$ is finite, from Theorem \ref{several density} we obtain a set $\Lambda \subset \R^n,$ \textcolor{black}{ which is a product of one dimensional closed sets of positive Lebesgue measure} such that the subspace $\Phi_\Lambda(\R^n)$ is not dense in $(C_\psi(\R^n),\|\cdot\|_\psi).$ So there exists a non-zero complex Borel measure $\mu \in (C_\psi(\R^n),\|\cdot\|_\psi)^*$ satisfying equation (\ref{eq;4}) (follows from Lemma \ref{dual lem}) such that $\textcolor{black}{T_\mu}(\Phi_{\Lambda})\equiv 0.$ This implies that $\widehat{\mu}$ vanishes on $\Lambda.$ Hence the theorem is proved.
	\end{proof}
	
	\begin{rem}
  Now, it is obvious that Theorem \ref{equiv} stated in the introduction follows immediately from Theorem \ref{several density} and Theorem \ref{several beurling}.
	\end{rem}
	
    	It clearly follows that the above result is true for $f\in L^p(\R^n),$ for $1 \leq p \leq 2$ instead of the complex measure $\mu.$ \textcolor{black}{Moreover,} using the duality of $L^p$ spaces and Theorem \ref{Lp dense}, in a similar manner as above \textcolor{black}{one can} get the following analogue of \textcolor{black}{Lemma}  \ref{psi-ber} on $\R^n$ for $f \in L^p(\R^n,\textcolor{black}{|}\mu\textcolor{black}{|}),$ $1 < p \leq \infty,$ where $\mu$ is a \textcolor{black}{complex (finite)} measure satisfying certain condition. \textcolor{black}{However, here we provide a straightforward proof using Theorem \ref{several beurling}.}
	\begin{Cor} \label{Beurling lp} 
		Let $\psi:[0,\infty)\rightarrow [0,\infty)$ be an increasing (not necessarily continuous) function such that 
		\be \label{psi int}\int_0^\infty\dfrac{\psi(x)}{1+x^2}dx=\infty\ee
		and $\mu$ be a \textcolor{black}{complex (finite) Borel} measure satisfying 
		\bes\int_{\R^n} e^{\psi(|x|)}d\textcolor{black}{|}\mu\textcolor{black}{|}(x)<\infty.\ees
		If $f\in L^p(\R^n,\textcolor{black}{|}\mu\textcolor{black}{|}),$ for $1<p\leq\infty,$ is such that $\textcolor{black}{\widehat{fd\mu}}$ vanishes on $\Lambda \subset \R^n,$ where $\Lambda$ is \textcolor{black}{ a product of one dimensional closed sets of positive Lebesgue measure}, then $f=0, ~\mu$ almost everywhere. 
	\end{Cor}
\begin{proof}
\textcolor{black}{Since $\mu$ is a finite measure, $f\in L^p(\R^n,|\mu|)$ implies $f\in L^1(\R^n,|\mu|)$ and hence it follows that $fd\mu$ is a complex measure on $\R^n.$ As $1<p\leq \infty,$ $p'$ (the dual index of $p$) is non-zero. Now, by H\"{o}lder's inequality, we have
\bes 
\int_{ \R^n} e^{\psi(|x|)/p'} |f(x)| \, d|\mu|(x)\leq \left(\int_{ \R^n} |f(x)|^p d|\mu|(x)\right)^{1/p}\left(\int_{ \R^n} e^{\psi(|x|)} d|\mu|(x)\right)^{1/p'}< \infty.
\ees
As $\psi/p'$ satisfies (\ref{psi int}) and $\widehat{fd\mu}$ vanishes on \textcolor{black}{ a product of one dimensional closed sets of positive Lebesgue measure,}  then from Theorem \ref{several beurling}, it follows that $fd\mu$ is identically zero, which further implies that $f=0,\,\mu$ almost everywhere.}
\end{proof}
	\begin{rem}
  Note that as an uncertainty principle the above result is somewhat unusual since the decay condition on the measure $\mu$ allows some growth on the function $f.$ 
	\end{rem}
	\begin{rem}
Taking $f$ to be the constant function $1$ in the above theorem, if $\widehat{\mu}$ vanishes on \textcolor{black}{ a product of one dimensional closed sets of positive Lebesgue measure}, then  it is easy to see that $\mu$ is identically zero. Hence this can be considered as a generalisation of Beurling's theorem (Theorem \ref{several beurling}).
	\end{rem}	

	\section{On the $n$-dimensional torus $\T^n$}
	
	In this section we will prove analogues of Theorem \ref{several density} and Theorem \ref{several beurling} on the $n$-dimensional torus
	$$\T^n=\left\{\left(e^{ix_1}, \cdots , e^{ix_n}\right):~ x_1, \cdots ,x_n\in [-\pi,\pi)\right\}.$$ Clearly we can identify $\T^n$ with $[-\pi,\pi)^n$ which is of finite Lebesgue measure in $\R^n.$ As in the case of $\R^n,$ functions vanishing on a set of positive Lebesgue measure inside $\T^n$ certainly makes sense. Moreover, $\T^n$ being a set of finite Lebesgue measure, it is enough to consider functions in $L^1(\T^n).$ For $f\in L^1(\T^n),$ we define its Fourier coefficients by the formula,
	$$ \widehat{f}(k)= \frac{1}{(2\pi)^n}\int_{\T^n}f(x)e^{-ix\cdot k}dx, \text{ for } k\in \Z^n.$$ In order to prove these results, we need to proceed as in the previous section. First, in a similar way as $\R^n,$ we will restate Theorem \ref{ber-th-c}:
\textcolor{black}{	
	\begin{Lem}\label{psi-ber-c}
		Let $\psi$ be an increasing non-negative function on $\N \cup \{0\}$ such that
  \bes \sum_{k=0}^\infty\dfrac{\psi(k)}{1+k^2} = \infty.
		\ees
 If $f\in L^2(\T)$ is such that 
		\bes
		\sum_{k=0}^\infty|\widehat{f}(k) ~ e^{\psi(k)} |^2 < \infty,  \ees
		then $f$ satisfies (\ref{eq;7c}), in particular Theorem \ref{ber-th-c} holds for $f.$ 
	\end{Lem}
}
	
	\subsection{Density of Exponentials}
	
	For any function $\psi:[0,\infty) \rightarrow [0,\infty)$ such that $\psi(x)\to \infty $ as $x\to \infty,$ we consider the space of sequences over $\Z^n$ defined by $$c_{\psi}(\Z^n)=\left\lbrace (a_k)_{k\in \Z^n} : \lim_{|k|\to \infty}\dfrac{a_k}{e^{\psi(|k|)}}=0\right\rbrace .$$ It is easy to see that $(c_{\psi}(\Z^n),\|\cdot \|_{\psi})$ is a normed linear space where 
	$$\|A\|_{\psi}=\sup_{k\in\Z^n}\dfrac{|a_k|}{e^{\psi(|k|)}}, \quad \txt{ for } A=(a_k)_{k\in \Z^n}\in c_{\psi}(\Z^n).$$ 
	We define the Banach space $(c_0(\Z^n),\|\cdot\|_{\infty})$ of sequences vanishing at infinity by 
	$$ c_0(\Z^n)=\ds{\left\lbrace (a_k)_{k\in\Z^n} : \lim_{|k|\to \infty}a_k=0 \right\rbrace }.$$
 \textcolor{black}{It is well known that the dual of $(c_0(\Z^n),\|\cdot\|_{\infty}) $ is $(l^1(\Z^n),\|\cdot\|_{1})$ where
		$$ l^1(\Z^n)=\left\{ (a_k)_{k\in\Z^n} : \sum_{k\in\Z^n}|a_k|<\infty \right\}.$$ }	As in the case of $\R^n,$ we will need to calculate the dual of $(c_{\psi}(\Z^n),\|\cdot \|_{\psi})$ explicitly. \textcolor{black}{In a similar manner, we define a bijective linear isometry $T_\psi: (c_0(\Z^n),\|\cdot\|_{\infty}) \to (c_{\psi}(\Z^n),\|\cdot\|_{\psi})$ by $T_\psi((a_k)_{k\in\Z^n}) = (e^{\psi(|k|)}a_k)_{k\in\Z^n},$ for $(a_k)_{k\in\Z^n}\in c_0(\Z^n).$ From this, we get a natural bijection between the respective dual spaces, which allows us to obtain the following lemma:}
	
	\begin{Lem}\label{dual-lem-c}
		$(c_{\psi}(\Z^n),\|\cdot\|_{\psi})$ is a Banach space isometrically isomorphic to $(c_0(\Z^n),\|\cdot\|_{\infty})$ and its dual is given by
		$$(c_{\psi}(\Z^n),\|\cdot\|_{\psi})^*=\left\lbrace (d_k)_{k\in \Z^n} \in l^1( \Z^n) :
		\sum_{k\in\Z^n}e^{\psi(|k|)}|d_k|<\infty \right\rbrace .$$
	\end{Lem}

	Given $\Lambda \subset \T^n,$ we consider the linear span of exponentials given by 
	$$\Phi_{\Lambda}(\Z^n)= span \lbrace E_{\la}: \la \in \Lambda \rbrace, \quad \txt{ where } E_\lambda=\left(e^{i\la \cdot k}\right)_{k\in\Z^n}.$$ 
	We will now prove the density of the subspace $\Phi_{\Lambda}(\Z)$ in $c_{\psi}(\Z)$ for certain $\Lambda \subset \T.$
	\begin{Thm}\label{density-c}
		Let $\psi:[0,\infty) \rightarrow [0,\infty)$ be an increasing function satisfying 
		\bes \sum_{k=0}^\infty\dfrac{\psi(k)}{1+k^2} = \infty. \ees
		For any $\Lambda \subset \T$ such that $\Lambda$ is of positive Lebesgue measure, $\Phi_{\Lambda}(\Z)$ is dense in $(c_{\psi}(\Z),\|\cdot \|_{\psi}).$
	\end{Thm}
	
	\begin{proof}
		We consider a bounded linear functional $L$ on $(c_{\psi}(\Z),\|\cdot \|_{\psi})$ which vanishes on the space $\Phi_{\Lambda}(\Z).$ From Lemma \ref{dual-lem-c}, we get a sequence $D=(d_k)_{k\in\Z} \in l^1(\Z)$ satisfying 
		\be \label{dk} \ds{\sum_{k\in\Z}e^{\psi(|k|)}|d_k|<\infty} \ee such that
		\be \label{dualrel}  L(B)=\sum_{k\in\Z} b_k~d_k, \quad \txt{ for } B=(b_k)_{k\in\Z} \in c_{\psi}(\Z). \ee
		Since $l^1(\Z)\subset l^2(\Z),$ we have $D \in l^2(\Z).$ So we get $g \in L^2(\T)$ such that $\widehat{g}(k) = d_k$ for all $k \in \Z.$ Since $L$ vanishes on $\Phi_{\Lambda}(\Z),$ using (\ref{dualrel}) we have $$L(E_\la) = \sum_{k\in\Z} e^{i\lambda \cdot k}d_k=0, \quad \txt{ for all } \lambda\in \Lambda .$$ It follows that $g(\la) = 0$ for almost every $\la \in \Lambda,$ a set of positive Lebesgue measure inside $\T.$ Moreover, from (\ref{dk}) we get that  
		$$ \sum_{k=0}^\infty e^{2\psi(k)}|\widehat{g}(k)|^2<\infty .$$
		So we can apply \textcolor{black}{Lemma} \ref{psi-ber-c} to $g,$ to conclude that $g\equiv 0.$ It follows that $d_k=0,\text{ for all } k \in \Z.$ Hence $L$ is identically zero and the result follows.
	\end{proof}
	
	\vspace{0.1in}
	
	Now, we will look at the $n$-dimensional case. We will prove that the density of  $\Phi_\Lambda(\Z^n)$ in $(c_\psi(\Z^n),\|\cdot\|_\psi),$  \textcolor{black}{where $\Lambda$ is a product of one dimensional sets of positive Lebesgue measure} is also characterized by the integrability condition on $\psi:$
	
	\vspace{0.1in}
	\begin{Thm}\label{sev-den-c}
		Let $\psi$ be a non-negative increasing function on $[0,\infty)$ such that $\psi(x)\to \infty $ as $x\to \infty.$ The space $\Phi_\Lambda(\Z^n)$ is a dense subspace in $(c_\psi(\Z^n),\|\cdot\|_\psi)$ for any set \textcolor{black}{$\Lambda = \Lambda_1\times \cdots \times \Lambda_n  \subset \T^n,$ where each $\Lambda_j$ is a set of positive Lebesgue measure} if and only if 
		\be \label{psiinf} \sum_{k=0}^\infty\dfrac{\psi(k)}{1+k^2} = \infty. \ee
	\end{Thm}

\begin{proof}
		First we will prove the density of $\Phi_\Lambda(\Z^n)$ in $(c_\psi(\Z^n),\|\cdot\|_\psi)$ assuming (\ref{psiinf}). We will reduce the problem to the case $n=1$ and then apply Theorem \ref{density-c} to get the result. Let us define, for $x\in [0,\infty),$ $$\psi_0 (x)= \dfrac{\psi (x)}{n}.$$		
		\textcolor{black}{We consider the space $c_{00}(\Z^n)$ of eventually zero sequences over $\Z^n$ given by  
		\bes c_{00}(\Z^n)=\left\{ (a_k)_{k\in \Z^n}: \exists \, m \geq 0 \txt{ such that } a_k=0 \txt{ for } |k|>m \right\}. \ees 
		Now, we define the following space of \textcolor{black}{sequences}
		\bes 
		\mathcal{P}c_{\psi_0}(\Z^n) = \text{span} \left\{ (b_k)_{k  =(k_1, \cdots, k_n) \in \Z^n} :  b_k=b^1_{k_1}\cdots b^n_{k_n}, (b^j_{m})_{m\in \Z}\in c_{\psi_0}(\Z), 1\leq j \leq n  \right\}.
		\ees
		}		
	 We will prove that $ \mathcal{P}c_{\psi_0}(\Z^n) \subseteq c_{\psi}(\Z^n).$ Since $\psi_0$ is increasing, for any $(b_k)_{k =(k_1, \cdots, k_n) \in \Z^n}\in \mathcal{P}c_{\psi_0}(\Z^n),$ we have
		$$\dfrac{|b_k|}{e^{\psi(|k|)}}=\dfrac{|b^1_{k_1}|\cdots |b^n_{k_n}|}{e^{n\psi_0(|k|)}}\leq \dfrac{|b^1_{k_1}|}{e^{\psi_0(|k_1|)}}\cdots \dfrac{|b^n_{k_n}|}{e^{\psi_0(|k_n|)}}.$$
		If $|k|\to \infty$ for $k \in \Z^n,$ there is $1\leq p \leq n$ such that $|k_p|\to \infty.$ Since $(b^p_{m})_{m\in \Z}\in c_{\psi_0}(\Z),$ we get that 
		$$\ds{\lim_{|k_p|\to \infty}\dfrac{|b^p_{k_j}|}{e^{\psi_0(|k_p|)}}=0}.$$ Also, $\ds{\left(\dfrac{|b^j_{m}|}{e^{\psi_0(|m|)}}\right)_{m\in \Z}}$ are bounded for all $1\leq j \leq n.$ So we conclude that $(b_k)_{k \in \Z^n}\in c_{\psi}(\Z^n).$ 
  
  Modifying the standard argument used to prove that $c_{00}$ is dense in $(c_0, \|\cdot\|_\infty),$ it is easy to see that $c_{00}(\Z^n)$ is dense in $(c_{\psi}(\Z^n),\|\cdot \|_{\psi}).$ \textcolor{black}{Moreover, since $c_{00}(\Z^n)\subset \mathcal{P}c_{\psi_0}(\Z^n) \subset c_\psi(\Z^n),$ it follows that $\mathcal{P}c_{\psi_0}(\Z^n)$ is dense in $(c_\psi(\Z^n),\|\cdot\|_\psi).$ }
		
		\textcolor{black}{In order to prove the result that $\Phi_\Lambda(\Z^n)$ is dense in $(c_\psi(\Z^n),\|\cdot\|_\psi),$ we will show that $\Phi_{\Lambda}(\Z^n)$ is dense in $(\mathcal{P}c_{\psi_0}(\Z^n),\|\cdot\|_\psi).$} For this, it is enough to consider $B = (b_k)_{k \in \Z^n} \in \mathcal{P}c_{\psi_0}(\Z^n)$ of the form $$b_k=b^1_{k_1}\cdots b^n_{k_n},\quad \txt{ for }  k=(k_1, \cdots , k_n)\in \Z^n,$$
		where $ B^j=(b^j_m)_{m\in\Z}\in c_{\psi_0}(\Z)$ for $1\leq j \leq n.$ Since $m(\Lambda_j)>0$ for each $j$ and (\ref{psiinf}) is true for $\psi_0$ as well, applying Theorem \ref{density-c} we get that $\Phi_{\Lambda_j}(\Z)$ is dense in $(c_{\psi_0}(\Z), \|\cdot\|_{\psi_0}).$ So for any $0<\epsilon<1,$ we get $E^j=(e^j_m)_{m\in\Z}\in \Phi_{\Lambda_j}(\Z)$  such that
		\begin{equation*}
			\|E^j - B^j\|_{\psi_0} = \sup_{m \in \Z} \frac{|b^j_{m}- e^j_{m}|}{e^{\psi_0(|m|)}}<\epsilon, \quad \txt{ for each } 1\leq j\leq n. 
		\end{equation*}
		By triangle inequality we have,
		$$ \|E^j\|_{\psi_0} < 1+\|B^j\|_{\psi_0}, \quad \txt{ for each } 1\leq j \leq n . $$ 
		\textcolor{black}{Now we define $E=(e_k)_{k \in \Z^n} \in \Phi_{\Lambda}(\Z^n)$ by}
		\beas e_k &=& e^1_{k_1}\cdots e^n_{k_n},\quad \txt{ for }  k=(k_1, \cdots , k_n)\in \Z^n, \\  
		\txt{ and } \quad \quad \quad \quad  e^0_{m} &=& e^{\psi_0(|m|)}=b^{n+1}_{m},\quad \txt{ for } m \in \Z.\eeas 

  Since $\psi = n \psi_0$ is increasing, for $k=(k_1,\cdots , k_n)\in \Z^n$ we get that
		\begin{eqnarray*}
			\frac{|b_k- e_k|}{e^{\psi(|k|)}} &\leq& \frac{|b^1_{k_1}\cdots b^n_{k_n}- e^1_{k_1}\cdots e^n_{k_n}|}{e^{\psi_0(|k_1|)}\cdots e^{\psi_0(|k_n|)}}\\
			&\leq & \sum_{j=1}^n\dfrac{|b^j_{k_j}-e^j_{k_j}|}{e^{\psi_0(|k_j|)}}\left( \prod_{p=0}^{j-1}\dfrac{|e^p_{k_p}|}{e^{\psi_0(|k_p|)}} \prod_{p=j+1}^{n+1}\dfrac{|b^p_{k_p}|}{e^{\psi_0(|k_p|)}} \right) \\ 
			&< & \epsilon ~ n \prod_{j=1}^n (1+\|B^j\|_{\psi_0}) \\ 
			&\leq & C\epsilon.
		\end{eqnarray*}
		So we get that $\|B-E\|_\psi < C \epsilon.$ This proves the first part.
		
		For the converse part, let us assume that $\psi:[0,\infty) \rightarrow [0,\infty)$ is an increasing function satisfying 
		\bes \sum_{k=0}^\infty\dfrac{\psi(k)}{1+k^2} < \infty. \ees 
		Since $\psi$ is increasing, it is easy to see that \bes \int_0^{\infty}\dfrac{\psi(x)}{1+x^2}dx<\infty. \ees 
		Now, we can choose $\epsilon>0$ and $a \in \T^n$ in such a way that $B(0,\epsilon)$ and $B(a,2\epsilon)$ are disjoint subsets inside $\T^n.$
		By Theorem 2.6 (b) of \cite{BRS}, we get a non-zero radial function $g\in C_c^\infty(\R^n)$ supported inside $B(0,\epsilon)$  satisfying 
		$$|\widehat{g}(\xi)|\leq Ce^{-\psi(|\xi|)}, \quad \text{ for all } \xi \in \R^n.$$ 
		We define  
		$$h(x)= (2 \pi)^n \sum_{m\in \Z^n}g(x+ 2 \pi m), \quad \txt{for } x\in \T^n.$$ 
		We note that the choice of $\epsilon$ implies that $\ds h=g|_{\T^n}$ and $h\in C^\infty(\T^n).$ Moreover, by Poisson summation formula, we have
		$$|\widehat{h}(m)|=|\widehat{g}(m)|\leq Ce^{-\psi(|m|)}, \quad \txt{ for all } m\in \Z^n.$$
		Now, let us define $f=h*h \in C^\infty(\T^n)$ which is non-zero (since $h$ is non-zero) and vanishes on $B(a,\epsilon)$ because $h$ vanishes on $B(a,2\epsilon).$ So we get that 
		$$ |\widehat{f}(m)|\leq Ce^{-\psi(|m|)}|\widehat{h}(m)|, \quad \txt{ for all } m\in \Z^n.$$
		Since $h\in C^\infty(\T^n),$ we obtain $\widehat{h} \in l^1(\Z^n)$ which implies that 
		$$\ds{\sum_{m\in \Z^n}|\widehat{f}(m)|e^{\psi(|m|)}<\infty.}$$
		Therefore by Lemma \ref{dual-lem-c}, we get that $\widehat{f}\in (C_\psi(\Z^n),\|\cdot\|_\psi)^*.$ 
		Moreover, since $f \in C^\infty(\T^n)$ vanishes on the set $\Lambda=B (a,\epsilon),$ we have $\widehat{f}$ vanishes on $\Phi_\Lambda(\Z^n).$ Since $f$ is non-zero, we can conclude that the space $\Phi_\Lambda(\Z^n)$ is not dense in $(C_\psi(\Z^n),\|\cdot\|_\psi)$ for some
		set $\Lambda,$ \textcolor{black}{which contains a product of one dimensional sets of positive Lebesgue measure.}
	\end{proof}

	\subsection{Beurling's Theorem}
	
	We will prove Beurling's result on the torus $\T^n.$ For this, we will need the following lemma regarding a set \textcolor{black}{which is a product of one dimensional sets of positive Lebesgue measure}.
	\begin{Lem}\label{positive RT}
		If $\Lambda\subset \T^n$ is \textcolor{black}{ a product of one dimensional sets of positive Lebesgue measure} and $E\subset \T^n$ is a set of zero Lebesgue measure, then $\overline{\Lambda \setminus E}$ \textcolor{black}{contains a product of one dimensional sets of positive Lebesgue measure}. 
	\end{Lem} 
	\begin{proof}
		Since $\Lambda\subset \T^n$ is \textcolor{black}{ a product of one dimensional sets of positive Lebesgue measure}, there exists $\Lambda_j\subset \T$ such that $m(\Lambda_j)>0,$ for each $1\leq j \leq n,$ and $\Lambda_1\times \cdots \times \Lambda_n \textcolor{black}{=}\Lambda.$ By Lebesgue density theorem, we get that
		\bes \lim_{r \ra 0} \frac{m\left( B(\lambda_j,r)\cap\Lambda_j \right)}{m( B(\lambda_j,r))} = 1, \quad \txt{ for almost every } \lambda_j\in\Lambda_j. 
		\ees	   
		It follows that there is a set $\Gamma_j \subset \Lambda_j$ satisfying $m(\Gamma_j)=m(\Lambda_j)$ such that for all $\lambda_j\in \Gamma_j,$ we have	 
		$$m\left( B(\lambda_j,r)\cap\Lambda_j \right)>0, \quad \txt{ for all }r>0.$$ 
		In order to prove the lemma, it is enough to show that $\Gamma_1\times \cdots \times \Gamma_n\subseteq \overline{\Lambda \setminus E}.$ For any $\gamma = (\gamma_1, \cdots , \gamma_n)\in \Gamma_1\times \cdots \times \Gamma_n$ and $r>0,$ we note that
		\beas B(\gamma,r)\cap (\Lambda \setminus E) &\supseteq& \Big( B\big(\gamma_1, n^{-\frac{1}{2}} r \big)\times \cdots \times B\big(\gamma_n,n^{-\frac{1}{2}} r \big) \Big) \cap (\Lambda \setminus E)\\
		&\supseteq& \Big[ \Big(B\big(\gamma_1,n^{-\frac{1}{2}} r\big)\cap \Lambda_1 \Big) \times \cdots \times \Big( B \big(\gamma_n, n^{-\frac{1}{2}} r \big) \cap\Lambda_n \Big)\Big] \setminus E,
		\eeas
		which is a set of positive Lebesgue measure in $\R^n.$ So it follows that any neighbourhood of $\gamma$ intersects $\Lambda \setminus E.$ Hence we can conclude that $\Gamma_1\times \cdots \times \Gamma_n \subseteq \overline{\Lambda \setminus E}.$ 
	\end{proof}
	
	\begin{rem}
		If $\Lambda\subset \T^n$ is \textcolor{black}{ a product of one dimensional sets of positive Lebesgue measure} and $E\subset \T^n$ is a set of zero Lebesgue measure, then $\Lambda \setminus E$ may not \textcolor{black}{contain a product of one dimensional sets of positive Lebesgue measure}. For example, if we consider the set 
		$$E=\lbrace (x,y)\in [0,1]\times [0,1] : x-y \in \mathbb{Q} \rbrace$$ 
		of zero Lebesgue measure inside the set $\Lambda = [0,1]\times [0,1],$ \textcolor{black}{ which is a product of one dimensional sets of positive Lebesgue measure,} then $\Lambda \setminus E$ \textcolor{black}{does not contain a product of one dimensional sets of positive Lebesgue measure} in $\T^2.$
	\end{rem}
	
	Now, we will use Theorem \ref{sev-den-c} to prove Beurling's theorem on $\T^n.$ We note here that the original result on the unit circle $\T$ (\textcolor{black}{Lemma} \ref{psi-ber-c}) was for functions in $L^2(\T).$ However, we have been able to improve the following result for functions in $L^1(\T^n).$
	
	\begin{Thm}\label{sev-fn-c}
		Let $\psi:[0,\infty)\to [0,\infty)$ be an increasing function such that $\psi(x) \ra \infty$ as $x \ra \infty$ and 
		\begin{equation*}
			S = \sum_{k=0}^\infty\dfrac{\psi(k)}{1+k^2}.
		\end{equation*}
		\begin{itemize}
			\item[(a)] \textcolor{black}{If $S = \infty$ and $f\in L^1(\T^n)$ is satisfying
   \begin{equation}\label{eq;10c}
				\sum_{k\in\Z^n}|\widehat{f}(k)|e^{\psi(|k|)}<\infty 
			\end{equation}
			 such that $f$ vanishes on a set $\Lambda =\Lambda_1\times\cdots\times\Lambda_n$ , where each $\Lambda_j$ is a set of positive Lebesgue measure in $\T,$} then $f$ is zero almost everywhere.
			\item[(b)] If \textcolor{black}{$S<\infty$,} then there is a non-zero $f\in C^\infty(\T^n)$ vanishing on \textcolor{black}{ a product of one dimensional sets of positive Lebesgue measure and} satisfying (\ref{eq;10c}).
		\end{itemize}
	\end{Thm}
	
	\begin{proof}
		First we shall prove (a). Since $f\in L^1(\T^n)$ and $(\widehat{f}(k))_{k\in\Z^n} \in l^1(\Z^n)$ (follows from (\ref{eq;10c})), the Fourier series of $f$ converges absolutely and uniformly to a continuous function $g$ on $\T,$ that is, 
		\be \label{fourierseries} g(\lambda)=\sum_{k\in \Z^n} \widehat{f}(k)e^{ik\cdot \lambda}, \quad \txt{ for all } \lambda\in \T^n. \ee
		Moreover, $f=g$ almost everywhere. So we get that 
		$$E=\{x\in \T^n:f(x)\neq g(x)\}$$ 
		is a set of zero Lebesgue measure and $g$ vanishes on $\Lambda\setminus E.$ Since $g$ is continuous, $g$ vanishes on $F=\overline{\Lambda\setminus E},$ which \textcolor{black}{contains a product of one dimensional sets of positive Lebesgue measure} by Lemma \ref{positive RT}. Now we will prove that $\widehat{f}(k)=0 \txt{ for all } k\in \Z^n.$ Since $(\widehat{f}(k))_{k\in\Z^n}$ satisfies (\ref{eq;10c}), from Lemma \ref{dual-lem-c} we get that $(\widehat{f}(k))_{k\in \Z^n}\in (c_{\psi}(\Z^n),\|\cdot \|_{\psi})^*,$ that is, there is a bounded linear functional $T_f$ on $c_{\psi}(\Z^n)$ given by
		\be \label{functional} T_f(A) = \sum_{k \in \Z^n} \widehat{f}(k) ~ a_k, \quad \txt{ for } A = (a_k)_{k \in \Z^n} \in c_{\psi}(\Z^n). \ee
		Now, as $g$ vanishes on $F,$ from (\ref{fourierseries}) and (\ref{functional}) we get that $T_f$ vanishes on the set $\Phi_{F}(\Z^n).$ Since $F$ \textcolor{black}{contains a product of one dimensional sets of positive Lebesgue measure} and $S=\infty$, we get from Theorem \ref{sev-den-c} that $\Phi_{F}(\Z^n)$ is dense in $(c_{\psi}(\Z^n),\|\cdot \|_{\psi}).$ So it follows that $T_f$ is identically zero. Hence we conclude that $\widehat{f}(k)=0 \txt{ for all } k\in \Z^n.$ This proves (a). 
		
		For part (b), if $S<\infty,$ the function $f$ constructed in the proof of converse part of Theorem \ref{sev-den-c} serves our purpose. This completes the proof.
	\end{proof}
	\begin{rem}
		There are sets of positive Lebesgue measure in $\T^n$ which do not differ from \textcolor{black}{ a product of one dimensional sets of positive Lebesgue measure} by a set of measure zero. For example, if we consider the positive measure set $E = \{(x,y)\in [0,1] \times [0,1]: x-y\in F\},$ where $F\subset [0,1]$ is the fat Cantor set (defined before), we can not obtain \textcolor{black}{ a product of one dimensional sets of positive Lebesgue measure} even by adding/subtracting a measure zero set to $E.$ 
	\end{rem}
	
	\noindent \textbf{Acknowledgement:} We wish to thank Professor Swagato K. Ray for several fruitful discussions and suggestions.

\end{document}